\documentclass[12pt,a4paper]{amsart}
\usepackage[colorlinks=true,         
            linkcolor=blue,
            citecolor=red,
            urlcolor=blue]{hyperref}
\usepackage{a4wide, amssymb,
amsmath, hyperref,mathrsfs,latexsym,stmaryrd,
mathrsfs,amsaddr,cancel,ifthen}
\usepackage{tikz}
\usepackage{tikz-cd}
\usetikzlibrary{matrix,arrows,decorations.pathmorphing}
\DeclareMathAlphabet{\mathpzc}{OT1}{pzc}{m}{it}

\DeclareMathOperator{\spec}{Spec}

\DeclareMathOperator{\rk}{rk}

\DeclareMathOperator{\num}{Num}

\DeclareMathOperator{\pic}{Pic}
\DeclareMathOperator{\sym}{Sym}

\newtheorem{thm}{Theorem}[section] 
\newtheorem{cor}[thm]{Corollary}
\newtheorem{lem}[thm]{Lemma} 
\newtheorem{prop}[thm]{Proposition}

\newtheorem{dfn}[thm]{Definition}
\theoremstyle{definition}

\theoremstyle{remark}
\newtheorem{rmk}[thm]{Remark}

\newtheorem{ex}[thm]{Example}

\numberwithin{equation}{section}

\newcommand{\blu}[1]{\textcolor{blue}{#1}}

 \title[Moduli of rank 2 Higgs sheaves  on elliptic surfaces]{Moduli of rank 2 Higgs sheaves \\[5pt] on elliptic surfaces}

\author{Ugo Bruzzo$^{\S\ast}$  and Vitantonio Peragine$^{\P}$}

\address{\small 
$^\S$Departamento de Matem\'atica, Universidade Federal da Para\'iba,  \\ Campus I, Jo\~ao Pessoa, PB, Brazil\\
$^\S$INFN (Istituto Nazionale di Fisica Nucleare), Sezione di Trieste, Italia\\
$^\S$IGAP (Institute for Geometry and Physics), Trieste, Italia\\
$^\S$Arnold-Regge Center for Algebra, Geometry   and Theoretical Physics, Torino, Italia \\[3pt] $^\P$ Scala al Monticello 4, 34126, Trieste, Italia \\[3pt] E-Mail: {\tt ugo.bruzzo@academico.ufpb.br, vitantonio.peragine@gmail.com}}

\date{October 6th, 2021} 
\thanks{U.B.'s research is partly supported by PRIN ``Geometria delle variet\`a algebriche", by Bolsa de Produtividade 313333/2020-3 of Brazilian CNPq, and INdAM-GNSAGA. He 
  is a member of the VBAC group. \\ \indent $^\ast$  On leave of absence from SISSA (Scuola Internazionale Superiore di Studi Avanzati), Trieste, Italia.} 
 
\begin{document} 
\begin{abstract} 
We study torsion-free, rank 2 Higgs sheaves on genus one fibered surfaces, (semi)stable with respect to suitable polarizations in the sense of Friedman and O'Grady. We prove that slope-semistability of a Higgs sheaf on the surface implies semistability on the generic fiber. In the case of Higgs sheaves of odd fiber degree on elliptic surfaces in characteristic $\neq 2$, we prove that any moduli space of Higgs sheaves with fixed numerical invariants splits canonically as the product of the moduli space of ordinary sheaves (with the same invariants), and the space of global regular $1$-forms on the surface. For elliptic surfaces with section in characteristic zero, and in the case arbitrary fiber degree, we prove that if a Higgs sheaf has reduced Friedman spectral curve, or is regular on a general fiber with non-reduced spectral cover, then its Higgs field takes values in the saturation of the pull-back of the canonical bundle of the base curve in the cotangent bundle of the surface.
\end{abstract}

\maketitle  
\setcounter{tocdepth}{1}
\thispagestyle{empty}


 \section{Introduction}
In a previous paper \cite{primo} we studied Higgs bundles $(V,\phi)$ on a class of elliptic surfaces $\pi:X\to B$, whose underlying vector bundle $V$ has vertical determinant and is fiberwise semistable. We proved that if the spectral curve\footnote{This 
is      Friedman's notion of {\em spectral curve}, see \cite{libro friedman}.}  of $V$ is reduced, then the Higgs field $\phi$ is \emph{vertical}, while if the bundle $V$ is fiberwise regular with reduced (resp., integral) spectral curve,
 and if its rank and second Chern number satisfy an inequality involving the genus of $B$ and the degree of the fundamental line bundle of $\pi$ (resp., if the fundamental line bundle is sufficiently ample), then $\phi$ is \emph{scalar}. These results were applied to the problem of characterizing slope-semistable Higgs bundles with vanishing discriminant on the class of elliptic surfaces considered, in terms of the semistability of their pull-backs via maps from arbitrary (smooth, irreducible, complete) curves to $X$; i.e., we partly established, for the class of elliptic surfaces considered, the conjecture
about Higgs bundles satisfying the last mentioned condition that was stated in \cite{BG1} and was studied in
\cite{BHR,BG1,BLL,BL,BC}.

In this paper we continue this study, mostly working in a slightly more general setting,  namely, we assume that
$\pi:X\to B$ is a {\em genus one fibered surface}, which means that the generic fiber $X_\eta$ is a genus one curve over $K$ (the function field of the generic point of $B$) which we do not assume to be smooth. 

In Section \ref{sec:suitable}, following \cite{friedman art, kieran}, whenever $S$ is  an integral, regular, complete surface, and
$c=(r,c_1,c_2)$ is a triple in $\Gamma_S=\mathbb{Z}\times\num(S)\times\mathbb{Z}$, we give a notion of {\em $c$-suitable polarization}.
We recall from \cite{huy} the following result: if   $c\in\Gamma_X$ is a numerical class, and $H$ a $c$-suitable polarization, then, for every torsion-free sheaf $F$ of class $c$ on $X$,   if $F$ is $\mu_H$-semistable, then $F_\eta$ is semistable;
if $F_\eta$ is stable, then $F$ is $\mu_H$-stable.  Here $F_\eta$ is the restriction of $F$ to the fiber of $X$ over the generic point $\eta$ of $B$.
Section \ref{sec:Higgs} is basically devoted to the extension of this result to rank 2 Higgs sheaves (Proposition \ref{prop:3.3}) on genus one fibered surfaces.

Section \ref{sec:odd} considers rank 2 Higgs sheaves of odd fiber degree over genus one fibered surfaces. A first result somehow generalizes to these surfaces what Franco et al.~proved for elliptic curves \cite{franco-et-al}, namely, the underlying sheaf of a semistable Higgs sheaf is stable, both for slope and Gieseker stability. This implies that the moduli spaces of semistable ordinary or Higgs sheaves coincide with the corresponding moduli spaces of {\em stable} objects. These identifications are established as isomorphisms of moduli schemes. Moreover,
there is a surjective scheme morphism of  moduli spaces $M_{\mathrm{Higgs}}(c)\to M(c)$ (hopefully the meaning of symbols is clear,
anyway they will be defined in Section \ref{sec:odd}). 

Another result in this Section is the following.
Assume the ground field $k$ is  of characteristic $\neq 2$, and the fibration $\pi:X\to B$ is elliptic and non-isotrivial. Let $F$ be a $\mu_H$-semistable, torsion-free sheaf   on $X$. Then any Higgs field on $F$ is necessarily {\em scalar}, i.e., it is given by the tensor product by a 1-form on $X$. This implies that the moduli space of Higgs sheaves with fixed numerical invariants splits canonically as the product of the moduli space of ordinary sheaves with the same invariants, and the space on global $1$-forms on the surface; in particular, it is smooth.

Section \ref{sec:arb} deals with   sheaves of arbitrary degree. The main result is as follows. Given an elliptic surface  $\pi:X\to B$ with a section and a torsion-free Higgs sheaf $(F,\phi)$  of  arbitrary rank  on $X$, one assumes that (1)  the restriction of $F$ to the generic fiber of $\pi$  is slope-semistable; (2) the Friedman spectral curve of $F$ is reduced; (3)
 the schemes of singularities of $F$ and $(\Omega_\pi)_{\textrm{t.f.}}$, where $\Omega_\pi$ is the relative cotangent sheaf,  are disjoint.
Then the Higgs field $\phi$ takes values in $\mathcal{S}_\pi$, the saturation of the pullback of the cotangent sheaf of $B$ in the
cotangent sheaf of $X$. This generalizes Corollary 4.3 in \cite{primo}. We also show that the same result holds if
the spectral curve is non-reduced when the rank is 2.

\bigskip
\section{Suitable polarizations on fibered surfaces}\label{sec:suitable}
In this section we work in the category of $k$-schemes, where $k$ is an arbitrary field. We want to recall the notion of suitable polarization on a fibered surface.

\subsection{Fibered surfaces}

A \emph{fibered surface} is a pair
$$
(X,\pi:X\to B),
$$
where $X$ (the \emph{total space}) and $B$ (the \emph{base}) are integral, regular, complete schemes of dimension $2$ and $1$, respectively, and $\pi$ (the \emph{projection}) is a surjective morphism whose generic fiber is regular and geometrically connected. As usual, we will denote by $X_b$ the fiber of $\pi$ over a point $b\in B$. Moreover, we will denote by $\eta$ the generic point of the base $B$, and by $K=\mathcal{O}_{B,\eta}=\kappa(\eta)$ its function field; thus $X_\eta$ is the \emph{generic fiber} of $\pi$. 

Let $\pi:X\to B$ be a fibered surface. It follows from the definition that $\pi$ is a proper and flat morphism; in particular, the generic fiber $X_\eta$ is a proper curve over $K$. Its genus is, by definition, the \emph{genus of the fibration}. Unless otherwise explicitly stated, we will not assume $X_\eta$ to be smooth over $K$, that is, if we denote by $\bar{K}$ an algebraic closure of $K$, then the \emph{general geometric fiber} $X_{\bar{\eta}}:=X_\eta\times_K\bar{K}$ of $\pi$ might well be singular. As an example, a genus one fibration is said to be \emph{elliptic} if its generic fiber is smooth over $K$; otherwise it is said to be \emph{quasi-elliptic}. In the latter case, by a theorem of J. Tate, the characteristic of the ground field $k$ is necessarily $2$ or $3$, and $X_{\bar{\eta}}$ is a cuspidal cubic over $\bar{K}$.

Let $\pic(X)$ be the Picard group of $X$, and $\num(X)$ its quotient by the subgroup consisting of isomorphism classes of invertible sheaves numerically equivalent to zero. Then the image in $\mathrm{Num}(X)$ of a closed fiber $X_b\hookrightarrow X$ of $\pi$ is independent on the choice of the point $b\in B(k)$; we will call it the \emph{fiber class of $\pi$}, and denote it by $\mathfrak{f}_\pi$, or just $\mathfrak{f}$.

\subsection{Suitable polarizations}

Let $S$ be an integral, regular, complete surface. The \emph{intersection form} on $\pic(S)$ and on $\num(S)$ will be denoted by $
\lambda\otimes\mu\mapsto(\lambda\cdot\mu)$. The \emph{self-intersection} of a class $\lambda$ will be shortened to $\left(\lambda^2\right)$. We also set
$$
\Gamma_S:=\mathbb{Z}\times\num(S)\times\mathbb{Z}.
$$
Elements $c\in\Gamma_S$ will be written as
$$
c=(r,c_1,c_2);
$$
in fact, for any coherent sheaf $F$ on $S$, one gets a well defined element $\mathrm{c}(F)$ of $\Gamma_S$, which we call the \emph{(numerical) class of $F$}, by setting
$$
\mathrm{c}(F):=\left(\rk F,\det F,\deg_S\left(\mathrm{c}_2(F)\cap[S]\right)\right).
$$
Here $\rk F$ is the \emph{rank} of $F$, $\det F\in\mathrm{Pic}(S)$ its \emph{determinant}, and $\deg_S\left(\mathrm{c}_2(F)\cap[S]\right)$ the \emph{degree} of the \emph{second Chern class} of $F$ (i.e., the \emph{second Chern number} of $F$).

For $c=(r,c_1,c_2)\in\Gamma_S$, we set 
$$
\Delta(c)=2rc_2-(r-1)\left(c_1^2\right)\in\mathbb{Z}.
$$
Thus, for a coherent sheaf $F$ of class $c$ on $S$, the integer $\Delta(c)$ is the \emph{Bogomolov number} of $F$, that is, the degree of the \emph{discriminant} of $F$.

Now let $\pi:X\to B$ be a fibered surface, with fiber class $\mathfrak{f}_\pi$. The \emph{fiber degree} of a line bundle $\lambda$ on $X$ (or of a class $\lambda\in\num(X)$) is the intersection number
$$
\deg_\pi(\lambda):=(\lambda\cdot\mathfrak{f}_\pi)\in\mathbb{Z}.
$$
In particular, for a coherent sheaf $F$ on $X$, the integer
$$
\deg_\pi(F):=\deg_\pi(\det F)
$$
will be called the \emph{fiber degree of $F$}. The choice of terminology is due to the fact that for $b\in B$ general, one has the equality
$$
\deg_\pi(F)=\deg_{X_b}(F_b),
$$
where, as usual, $F_b:=F\otimes_{\mathcal{O}_X}\mathcal{O}_{X_b}$ is the pull-back of $F$ to $X_b$ along the canonical morphism $X_b\to X$, and 
$$
\deg_{X_b}(F_b):=\chi(X_b,F_b)-\mathrm{rk}(F_b)\chi(X_b,\mathcal{O}_{X_b})$$
is the degree of the coherent sheaf $F_b$ on the integral, complete $\kappa(b)$-curve $X_b$.

Following \cite{friedman art, kieran} we introduce a class of polarizations  on fibered surfaces that enjoy
particularly nice properties.

\begin{dfn}
Let $\pi:X\to B$ be a fibered surface with fiber class $\mathfrak{f}_\pi$, and let $c=(r,c_1,c_2)\in \Gamma_X$ be a numerical class. A polarization $H$ on $X$ is said to be \emph{$c$-suitable} if for all $\xi\in\num(X)$ satisfying
$$
-\frac{r^2}{4}\Delta(c)\leqslant \left(\xi^2\right)\qquad\mbox{and}\qquad (\xi\cdot\mathfrak{f}_\pi)\neq 0,
$$
one has
$$
(\xi\cdot H)(\xi\cdot\mathfrak{f}_\pi)>0.
$$
\end{dfn}

The next Proposition shows that $c$-suitable polarizations exist for any class $c$; a proof can be found, e.g., in the book \cite{huy}.

\begin{prop}
Let $X\to B$ be a fibered surface with fiber class $\mathfrak{f}$, and let $c\in\Gamma_X$ be a numerical class. Then, for any polarization $H_0$ on $X$, the class
$$
H_n:=H_0+n\mathfrak{f}\in\num(X)
$$
is a $c$-suitable polarization for all sufficiently big $n\in\mathbb{Z}$.
\end{prop}

\subsection{Stability with respect to suitable polarizations}

The next result, for a proof of which we again refer the reader to \cite{huy}, clarifies the usefulness of suitable polarizations. Let us begin by recalling the notion of slope-stability for torsion-free sheaves: let $(Y,H)$ be an integral, regular, polarized $k$-scheme, and let $F$ be a torsion-free sheaf on $Y$\footnote{That is, the stalk $F_y$ of $F$ at any point $y\in Y$ is a torsion-free $\mathcal{O}_{Y,y}$-module.}. 
Then the \emph{slope}  of $F$ with respect to the polarization $H$ is the rational number
$$
\mu_H(F):=\frac{\deg_H(F)}{\rk F},
$$
where
$$
\deg_H(F):=\deg_Y\left(\mathrm{c}_1(F)\cup H^{\dim(Y)-1}\cap[Y]\right)\in\mathbb{Z}
$$
is the \emph{$H$-degree of $F$},
and $F$ is said to be \emph{slope-semistable with respect to $H$}, or, more concisely, $\mu_H$-semistable, if, for each non-zero submodule $S$ of $F$, the inequality
$$
\mu_H(S)\leqslant\mu_H(F)
$$
holds; if, instead, one has the strict inequality
$$
\mu_H(S)<\mu_H(F)
$$
for all non-zero, proper submodules $S$ of $F$, then $F$ is said to be \emph{slope-stable with respect to $H$}, or $\mu_H$-stable. If $Y$ is a \emph{curve}, slope-(semi)stability will be understood with respect to the polarization defined by an arbitrary closed point of $Y$ (different closed points yield the same polarization). Moreover, in this case, the previous definitions make sense, using the appropriate notion of degree, even 
when $Y$ is singular.

\begin{prop}
\label{semistab wrt suitable}
Let $\pi:X\to B$ be a fibered surface. In addition, let $c\in\Gamma_X$ be a numerical class, and $H$ a $c$-suitable polarization. Then, for every torsion-free sheaf $F$ of class $c$ on $X$, the following implications hold:
\begin{enumerate}
\item if $F$ is $\mu_H$-semistable, then $F_\eta$ is semistable;
\item if $F_\eta$ is stable, then $F$ is $\mu_H$-stable.
\end{enumerate}

\end{prop}

We recall that the semistability of $F_\eta$ is equivalent to the existence of a dense open subscheme $U$ of $B$ such that the pull-back $F_b$ of $F$ to the curve $X_b$ is a semistable sheaf on $X_b$ for all closed points $b\in U$, while the stability of $F_\eta$ is equivalent to the existence of a closed point $b\in B$ such that the restriction $F_b$ is stable. Thus  Proposition \ref{semistab wrt suitable} says that, for a torsion-free sheaf $F$ on a fibered surface $X\to B$, the semistability of $F$ with respect to a $\mathrm{c}(F)$-suitable polarization implies that the restrictions of $F$ to almost all (i.e., all but  finite number of) closed fibers are themselves semistable; and that the existence of a single closed point $b\in B$ such that the restriction of $F$ to the fiber over $b$ is stable, is enough to guarantee the stability of $F$ with respect to any $\mathrm{c}(F)$-suitable polarization. We remark that the reverse implication in $(1)$ is \emph{false}. In fact, a torsion-free sheaf $F$ on $X$ might very well be (strictly) semistable on all integral closed fibers of $X\to B$, and yet be unstable (that is, non-semistable) with respect to every $\mathrm{c}(F)$-suitable polarization.

\bigskip
\section{Rank 2 Higgs sheaves on genus one fibered surfaces}\label{sec:Higgs}

In this section we keep working in the category of $k$-schemes, where $k$ is a field. We will show that Proposition \ref{semistab wrt suitable} can be generalized to the case of \emph{rank $2$ Higgs sheaves} on \emph{genus one} fibered surfaces. In the proof we shall need a slight generalization of a result from \blu{\cite{BL},}  stating that for torsion-free Hitchin pairs, with values in a slope-semistable locally free sheaf of non-positive degree, slope-semistability as an ordinary sheaf and as a pair are, in fact, equivalent; we include a proof of the result (Proposition \ref{non pos}), for the sake of completeness and for the reader's convenience. Let us recall, first of all, the notions of Hitchin pair, and of their slope-stability. Quite generally, let $Y$ be a scheme and $V$ a coherent sheaf on $Y$ (the \emph{value sheaf}). A \emph{$V$-valued (Hitchin) pair} on $Y$ is a pair
$(F,\phi)$, where $F$ is a coherent sheaf on $Y$, and
$$
\phi:F\to F\otimes_{\mathcal{O}_Y} V
$$
a morphism of $\mathcal{O}_Y$-modules. A subsheaf $\iota:S\hookrightarrow F$ is said to by \emph{$\phi$-invariant} if the restriction of $\phi$ to $S$ factors through $\iota\otimes 1:S\otimes V\to F\otimes V$. A $V$-valued pair $(F,\phi)$ is said to be \emph{integrable} if the composition
\begin{equation}
\label{fifi}
F\xrightarrow{\phi} F\otimes V\xrightarrow{\phi\otimes 1} F\otimes V\otimes V\xrightarrow{1\otimes q} F\otimes \wedge^2V,
\end{equation}
where $q:V\otimes V\to\wedge^2V$ is the canonical epimorphism, vanishes. The map \eqref{fifi} is also denoted by $\phi\wedge\phi$. If $V=\Omega_Y$, the sheaf of K\"ahler differentials of $Y$, an integrable $V$-valued pair is more commonly called a \emph{Higgs sheaf}. 

Let now $(Y,H)$ be a polarized $k$-scheme, with $Y$ integral and regular. In addition, let $V$ be a coherent sheaf on $Y$, and $\mathcal{F}=(F,\phi)$ a $V$-valued pair, with $F$ torsion-free (we also say in this case that the \emph{pair} is torsion-free). The pair $\mathcal{F}$ is said to be \emph{$\mu_H$-semistable} (resp., \emph{$\mu_H$-stable}) if, for each non-zero, proper, $\phi$-invariant subsheaf $S$ of $F$, one has the inequality $\mu_H(S)\leqslant\mu_H(F)$ (resp., $\mu_H(S)<\mu_H(F)$). As usual, we will extend the previous definition to the case in which $Y$ is a \emph{singular curve}.

\begin{prop}
\label{non pos}
Let $(Y,H)$ be a polarized $k$-scheme, with $Y$ integral and regular. Moreover, let $V$ be a non-zero, locally free, $\mu_H$-semistable sheaf on $Y$ satisfying
$$
\deg_H(V)\leqslant 0,
$$
and let $\mathcal{F}=(F,\phi)$ be $V$-valued, torsion-free pair. Then the following are equivalent:
\begin{enumerate}
\item the pair $\mathcal{F}$ is $\mu_H$-semistable;
\item the underlying sheaf $F$ of $\mathcal{F}$ is $\mu_H$-semistable.
\end{enumerate}

\end{prop}

\begin{proof}
Clearly, if the underlying sheaf $F$ of $\mathcal{F}$ is $\mu_H$-semistable, then so is $\mathcal{F}$. Let us assume now that $F$ is $\mu_H$-unstable, and prove that then so is $\mathcal{F}$. Let $M$ be the maximal $\mu_H$-destabilizing subsheaf of $F$, that is, the unique saturated subsheaf of $F$ such that:
\begin{enumerate}
\item $\mu_H(S)\leqslant\mu_H(M)$ for all subsheaves $S$ of $F$;
\item for any subsheaf $S$ of $F$, if $\mu_H(S)\geqslant\mu_H(M)$, then $S$ is a subsheaf of $M$
\end{enumerate}
(equivalently, $M$ is the smallest, with respect to inclusion, non-zero subsheaf in the $\mu_H$-Harder-Narasimhan filtration of $F$). We have $\mu_H(M)>\mu_H(F)$, since $F$ has been assumed $\mu_H$-unstable. Thus, if we can show that $M$ is $\phi$-invariant, then it will follow that the pair $\mathcal{F}$ is $\mu_H$-unstable, as claimed. To this end, we observe, first of all, that by the $\mu_H$-semistability of the value bundle $V$, the maximal destabilizing subsheaf of the (torsion-free) sheaf $F\otimes V$ is $M\otimes V$. Thus, to prove the inclusion $\psi(M)\subseteq M\otimes V$ it is enough to show that
$$
\mu_H(\psi(M))\geqslant\mu_H(M\otimes V).
$$
But $\psi(M)$ is a quotient of the $\mu_H$-semistable sheaf $M$, whence
$$
\mu_H(\psi(M))\geqslant\mu_H(M)\geqslant\mu_H(M)+\mu_H(V)=\mu_H(M\otimes V)
$$
(in the last inequality we made use of the assumption $\deg_H(V)\leqslant 0$), as claimed. 
\end{proof}

\begin{ex}
The previous result applies, for instance, to Higgs bundles on (integral, regular, complete) curves of genus $\leqslant 1$, and on projective spaces of any dimension.
\end{ex}

We are now ready to state and prove the generalization of Proposition \ref{semistab wrt suitable} to rank $2$ Higgs sheaves on genus one fibered surfaces. 

\begin{prop} \label{prop:3.3}
\label{gen}
Let $\pi:X\to B$ be a \emph{genus one} fibered surface, $c\in\Gamma_X$ a numerical class of the form
$$
c=(2,c_1,c_2),
$$
and $H$ a $c$-suitable polarization on $X$. Moreover, let $
\mathcal{F}=(F,\phi)$
be a torsion-free Higgs sheaf of class $c$ on $X$. Assume the Higgs sheaf $\mathcal{F}$ to be $\mu_H$-semistable. Then the pull-back of $F$ to the generic fiber of $\pi$ is semistable.
\end{prop}

In the proof, we shall need the next:

\begin{lem}
\label{diff kal}
Let $\pi:X\to B$ be a genus one fibered surface. Then the pull-back of the sheaf of K\"ahler differentials on $X$ to the generic fiber $X_\eta$ of $\pi$ is an extension of free invertible sheaves; in particular it is a locally free and (strictly) slope-semistable sheaf of rank two and trivial determinant (hence of degree and slope zero) on the $K$-curve $X_\eta$.
\end{lem}

\begin{proof}
We start by remarking that the sequence of relative K\"ahler differentials of $\pi$
$$
\mathcal{K}_\pi:\qquad 0\to\pi^\ast\omega_B\xrightarrow{\pi^\ast}\Omega_X\to\Omega_\pi\to 0
$$
is exact. In fact, first of all, the exactness of the sequence
$$
f^\ast\Omega_T\xrightarrow{f^\ast}\Omega_S\to\Omega_f\to 0
$$
holds for any morphism of schemes $f:S\to T$. When applied to $f=\pi$, this property shows, in addition, that the map $\pi^\ast:\pi^\ast\omega_B\to\Omega_X$ has image of rank one, and thus kernel of rank zero. But the sheaf $\pi^\ast\omega_B$ is invertible on the integral scheme $X$, thus torsion-free. It follows that the sheaf $\ker(\pi^\ast)$ is necessarily zero, showing exactness of $\mathcal{K}_\pi$ at $\pi^\ast\omega_B$ too.

Now, the morphism $\eta:\spec(K)\to B$ is flat, hence so is the projection 
$$
X_\eta=X\times_{\pi,B,\eta}\spec(K)\to X.
$$
Thus the pull-back of the exact sequence $\mathcal{K}_\pi$ along $X_\eta\to X$, that is, 
$$
\left(\mathcal{K}_\pi\right)_\eta:\qquad 0\to \omega_B(\eta)\otimes_K\mathcal{O}_{X_\eta}\to\left(\Omega_X\right)_\eta\to\Omega_{X_\eta}\to 0,
$$
is also exact. Here $\omega_B(\eta)\otimes_K\mathcal{O}_{X_\eta}$ is the free sheaf on $X_\eta$ with fiber the $1$-dimensional $K$-vector space $\omega_B(\eta)=\omega_B\otimes_{\mathcal{O}_B}K$, and the sheaf $\Omega_{X_\eta}$ is also free of rank $1$, since the $K$-scheme $X_\eta$ is a complete, integral, regular curve of genus one. This shows that $(\Omega_X)_\eta$ is an extension of free invertible sheaves, as claimed. The remaining properties of $(\Omega_X)_\eta$ (rank, determinant, degree, slope, strict semistability) follow easily from what we just proved.
\end{proof}

\begin{proof}[Proof of Proposition \ref{gen}]
We will show that if the pull-back $F_\eta$ of $F$ to $X_\eta$ is unstable, then the Higgs sheaf $\mathcal{F}$ is $\mu_H$-unstable. By Lemma \ref{diff kal} and  Proposition \ref{non pos}, the instability of $F_\eta$ implies the instability of the $\left(\Omega_X\right)_\eta$-valued pair
$$
\left(F_\eta,\phi_\eta:=\phi\otimes\mathcal{O}_{X_\eta}:F_\eta\to F_\eta\otimes\left(\Omega_X\right)_\eta\right).
$$
So let $\lambda\subset F_\eta$ be an invertible saturated subsheaf which is $\phi_\eta$-invariant and satisfies
\begin{equation}
\label{ineq}
\mu_{X_\eta}(\lambda)>\mu_{X_\eta}(F_\eta)=\frac{\deg_\pi(F)}{2}.
\end{equation}
Here we denoted by $\mu_{X_\eta}$ the slope function on $X_\eta$. By flat descent, $\lambda$ can be extended to a subsheaf $S$ of $F$, which can, and will, be assumed saturated. 

The sheaf $S$ is $\phi$-invariant. In fact the composition
$$
S\hookrightarrow F\xrightarrow{\phi} F\otimes\Omega_X\to (F/S)\otimes\Omega_X
$$
is zero, since it has torsion-free source and target, and it pulls-back on $X_\eta$ to
$$
\lambda\hookrightarrow F_\eta\xrightarrow{\phi_\eta} F_\eta\otimes\left(\Omega_X\right)_\eta\to (F_\eta/\lambda)\otimes\left(\Omega_X\right)_\eta,
$$
which is the zero map, since $\lambda\subset F_\eta$ is $\phi_\eta$-invariant. Thus, by the exactness of the sequence
$$
0\to S\otimes\Omega_X\to F\otimes\Omega_X\to(F/S)\otimes\Omega_X,
$$
the restriction $S\hookrightarrow F\xrightarrow{\phi} F\otimes\Omega_X$ of the Higgs field $\phi$ of $F$ to $S$ factors through the inclusion $S\otimes\Omega_X\hookrightarrow F\otimes\Omega_X$.

In addition, one has $
\mu_H(S)>\mu_H(F)$, namely,
\begin{equation}
\label{ineq2}
\left((2\mathrm{c}_1(S)-\mathrm{c}_1(F))\cdot H\right)>0.
\end{equation}
To show this, we start by using the equality 
$$
\mu_{X_\eta}(\lambda)=\mu_{X_\eta}(S_\eta)=\deg_\pi(S)
$$
to rewrite \eqref{ineq} as
$$
\deg_\pi(2\mathrm{c}_1(S)-\mathrm{c}_1(F))>0.
$$
Thus, \eqref{ineq2} will follow from the fact that the polarization $H$ is $c$-suitable, if we can show that
$$
-\Delta(c)\leqslant\left((2\mathrm{c}_1(S)-\mathrm{c}_1(F))^2\right).
$$
Following \cite{huy}, this can be seen as follows: one fixes isomorphisms
$$
S\simeq\alpha\otimes\mathcal{I},\qquad F/S\simeq\beta\otimes\mathcal{J},
$$
where $\alpha,\beta$ are invertible sheaves on $X$, and $\mathcal{I},\mathcal{J}$ ideal sheaves of closed subschemes of $X$ of dimension $\leqslant 0$, which allows one to write the Chern classes of $F$ as
\begin{eqnarray*}
\mathrm{c}_1(F)&=&\mathrm{c}_1(\alpha)+\mathrm{c}_1(\beta)=\mathrm{c}_1(S)+\mathrm{c}_1(F/S),\\
\mathrm{c}_2(F)&=&\mathrm{h}^0(\mathcal{O}_X/\mathcal{I})+(\alpha\cdot\beta)+\mathrm{h}^0(\mathcal{O}_X/\mathcal{J}).
\end{eqnarray*}
Thus one finds for $\Delta(c)$ the inequality
\begin{eqnarray*}
-\Delta(c)&=&\left(\mathrm{c}_1(F)^2\right)-4\mathrm{c}_2(F)\\
&=&
\left(\alpha^2\right)+2(\alpha\cdot\beta)+\left(\beta^2\right)\\
&&
-4(\alpha\cdot\beta)-4\left(\mathrm{h}^0(\mathcal{O}_X/\mathcal{I})+\mathrm{h}^0(\mathcal{O}_X/\mathcal{J})\right)\\
&\leqslant &\left((\mathrm{c}_1(\alpha)-\mathrm{c}_1(\beta))^2\right)\\
&=&
\left((2\mathrm{c}_1(S)-\mathrm{c}_1(F))^2\right).
\end{eqnarray*}

\end{proof}

\bigskip
\section{Odd fiber degree} 
\label{sec:odd}

In this section we work over an \emph{algebraically closed} field $k$. We fix, \emph{once and for all}, the following data:
\begin{enumerate}
\item a genus one fibered surface $\pi:X\to B$, with fiber class $\mathfrak{f}_\pi\in\num(X)$;
\item a numerical class $
c\in\Gamma_X=\mathbb{Z}\times\num(X)\times\mathbb{Z}$
of the form
$$
c=(2,c_1,c_2)
$$ on $X$, satisfying the following \emph{assumption}:
\begin{equation}
\label{odd}
\text{the fiber degree $\deg_\pi(c_1)=(c_1\cdot\mathfrak{f}_\pi)$ is odd}.
\end{equation}
\item a $c$-suitable polarization $H$ on $X$.
\end{enumerate}

 We are interested in torsion-free Higgs sheaves of class $c$ on $X$, slope-semistable (or semistable) with respect to $H$, and in their moduli. By \emph{semistability} without any further qualifier we mean \emph{Gieseker-semistability}, whose definition we proceed to recall for the sake of completeness: let $(Y,H)$ be a polarized scheme over an arbitrary field $K$, and let $F$ be a coherent sheaf on $Y$. The \emph{Hilbert polynomial} of $F$ with respect to the polarization $H$ is the unique polynomial $\mathrm{P}_{H}(F)\in\mathbb{Q}[T]$ such that 
$$
\mathrm{P}_{H}(F)(n)=\chi(Y,F\otimes H^n)
$$
for all $n\in\mathbb{Z}$. It is known that $\mathrm{P}_{H}(F)$ has degree equal to the dimension $\dim (F)$ of $F$\footnote{This is, by definition, the dimension of the support $\{y\in Y:F_y\neq 0\}$ of $F$.}. If $F$ in non-zero of dimension $d$, and $\alpha_{H,d}(F)\in\mathbb{Q}^\times$ is the leading coefficient of $\mathrm{P}_{H}(F)$, then the quotient
$$
\mathrm{p}_{H}(F):=\frac{\mathrm{P}_{H}(F)}{\alpha_{H,d}(F)}\in\mathbb{Q}[T]
$$
 is called the \emph{normalized} Hilbert polynomial of $F$ with respect to $H$. $F$ is said to be \emph{pure} if, for each non-zero subsheaf $S$ of $F$, one has the equality $\dim S=\dim F$. Finally, $F$ is said to be \emph{semistable with respect to $H$}, or $H$-semistable for short, if:
 \begin{enumerate}
 \item $F$ is pure;
 \item for each non-zero, proper subsheaf $S$ of $F$ one has the inequality
 \begin{equation}
 \label{semis}
 \mathrm{p}_{H}(S)\leqslant\mathrm{p}_{H}(F),
 \end{equation}
 where polynomials in $\mathbb{Q}[T]$ are ordered lexicographically.
 \end{enumerate}
As usual, if one replaces the inequality \eqref{semis} with the strict inequality $ \mathrm{p}_{H}(S)<\mathrm{p}_{H}(F)$, one obtains the notion of \emph{stability with respect to $H$}. If $F$ is torsion-free, the following chain of implications holds:
\begin{equation}
\label{chain}
\text{$F$ is $\mu_H$-stable $\Rightarrow$ $F$ is $H$-stable $\Rightarrow$ $F$ is $H$-semistable $\Rightarrow$ $F$ is $\mu_H$-semistable.}
\end{equation}

  If $V$ is a coherent sheaf on $Y$, \emph{a $V$-valued Hitchin pair $(F,\phi)$ is semistable} if $F$ is pure, and the inequality \eqref{semis} holds for all $\phi$-invariant, proper, non-zero subsheaves $S$ of $F$; in an analogous way one obtains the notion of \emph{stable pair}. If $F$ is torsion-free, the chain of implications obtained by replacing $F$ with $(F,\phi)$ in \eqref{chain} is also true.
  
The first consequence of Proposition \ref{gen}, combined with assumption \eqref{odd} is the next result, showing that for a rank $2$, torsion-free Higgs sheaf $(F,\phi)$ of odd fiber degree on a genus one fibered surface, the (slope-)semistability with respect to a $\mathrm{c}(F)$-suitable polarization is equivalent to the (slope-)stability of the underlying sheaf. More precisely (recall that the surface $X$, the class $c$, and the polarization $H$ are those specified at the beginning of the section):
  
 \begin{prop}
 \label{ss implies s}
 Let $\mathcal{F}=(F,\phi)$ be a torsion-free Higgs sheaf of class $c$ on $X$. Assume  $\mathcal{F}$ to be $\mu_H$-semistable. Then the underlying sheaf  $F$ of $\mathcal{F}$ is $\mu_H$-stable. Analogously, if $\mathcal{F}$ is $H$-semistable, then $F$ is $H$-stable.
\end{prop}

\begin{proof}
By Proposition \ref{gen}, we know that $F_\eta$ is a semistable locally free sheaf.
Its rank
$$
\mathrm{rk}(F_\eta)=\mathrm{rk}(F)=2
$$
and degree
$$
\deg_{X_\eta}(F_\eta)=\deg_\pi(F)
=\deg_\pi(c_1)
$$
are relatively prime. Thus, by well known properties of semistable vector bundles on curves, $F_\eta$ is actually stable. By Proposition \ref{semistab wrt suitable}, it then follows that $F$ is $\mu_H$-stable, as claimed. The last statement is a consequence of the first, and the analogue of \eqref{chain} for Higgs sheaves.
\end{proof}

Next, we will point out some consequences of Proposition \ref{ss implies s} for the moduli schemes of $H$-semistable Higgs (and ordinary) sheaves of class $c$. From this point on, we will denote by 
$$
M(c):=\mathrm{M}(X,H;c),\qquad M_{\mathrm{Higgs}}(c):=\mathrm{M}_{\mathrm{Higgs}}(X,H;c)
$$
the moduli schemes of $H$-semistable ordinary and Higgs sheaves on $X$ of class $c$, and by
$$
M^{\mathrm{s}}(c)=\mathrm{M}^{\mathrm{s}}(X,H;c),\qquad M_{\mathrm{Higgs}}^{\mathrm{s}}(c)=\mathrm{M}_{\mathrm{Higgs}}^{\mathrm{s}}(X,H;c)
$$
the open subschemes  of $M(c)$ and $M_{\mathrm{Higgs}}(c)$, corresponding to $H$-stable ordinary and Higgs sheaves, respectively. Then we have:

\begin{prop}
\label{res}
\begin{enumerate}
\item The open immersions
$$
M^\mathrm{s}(c)\hookrightarrow M(c),\qquad M^{\mathrm{s}}_{\mathrm{Higgs}}(c)\hookrightarrow M_{\mathrm{Higgs}}(c)
$$
are isomorphisms. In particular, closed points of $M(c)$ (resp., of $M_{\mathrm{Higgs}}(c)$) correspond bijectively to isomorphism classes of $H$-stable sheaves (resp., Higgs sheaves) of class $c$;
\item the operation of forgetting the Higgs field of an $H$-semistable Higgs sheaf of class $c$ gives a well-defined and surjective morphism of schemes
\begin{equation}
\label{forget}
p:M_{\mathrm{Higgs}}(c)\to M(c).
\end{equation}
\end{enumerate}

\end{prop}

\begin{proof}
\begin{enumerate}
\item Saying that the open immersion $M^{\mathrm{s}}_{\mathrm{Higgs}}(c)\hookrightarrow M_{\mathrm{Higgs}}(c)$ is an isomorphism is the same as saying that each $H$-semistable Higgs sheaf of class $c$ is actually $H$-stable, and this follows immediately from Proposition \ref{ss implies s}. In fact, if a Higgs sheaf $\mathcal{F}=(F,\phi)$ of class $c$ is $H$-semistable, then it is pure of dimension $2=\dim X$, namely torsion-free, so by Proposition \ref{ss implies s} its underlying sheaf $F$ is $H$-stable. This obviously implies that $\mathcal{F}$ is also $H$-stable, as claimed. Applying the same reasoning to Higgs sheaves of the form $(F,0)$, one obtains the analogous result for the moduli schemes of ordinary sheaves.
\item We recall that the scheme $M(c)$ corepresents the functor
$$
\varphi:(\mathrm{Sch}/k)^{\mathrm{op}}\to(\mathrm{Set})
$$
which assigns to a $k$-scheme $S$ (separated and of finite-type) the set $\varphi(S)$ of isomorphism classes of flat families of $H$-semistable sheaves of class $c$ on $X$ parametrized by $S$; and to a morphism of schemes $f:S\to T$ the map of sets $\varphi(T)\to\varphi(S)$ sending the isomorphism class of a family $\mathcal{G}\in\mathrm{Coh}(T\times X)$ to the isomorphism class of the family $(f\times\mathrm{id}_X)^\ast\mathcal{G}$. This means that there exists a natural transformation $\varphi\to \mathrm{h}_{M(c)}$, where $\mathrm{h}_{M(c)
}:(\mathrm{Sch}/k)^{\mathrm{op}}\to(\mathrm{Set})$ is the functor of points of the scheme $M(c)$, which is initial in the category of natural transformations of the form $\varphi\to\mathrm{h}_Y$, $Y$ being any $k$-scheme.

Analogously, the scheme $M_{\mathrm{Higgs}}(c)$ corepresents the functor
$$
\varphi_{\mathrm{Higgs}}:(\mathrm{Sch}/k)^{\mathrm{op}}\to(\mathrm{Set})
$$
sending a $k$-scheme $S$ to the set of isomorphism classes of flat families of $H$-semistable Higgs sheaves of class $c$ on $X$ parametrized by $S$.

Now, let $S$ be a $k$-scheme, and let $(\mathcal{F},\Phi)$ be a flat family of $H$-semistable Higgs sheaves of class $c$ on $X$ parametrized by $S$; thus, $\mathcal{F}$ is a coherent sheaf on $S\times X$, flat over $S$, and $\Phi:\mathcal{F}\to\mathcal{F}\otimes_{\mathcal{O}_{S\times X}}\Omega_X$ an $\mathcal{O}_{S\times X}$-linear map satisfying $\Phi\wedge\Phi=0$ in $\mathrm{Hom}_{\mathcal{O}_{S\times X}}(\mathcal{F},\mathcal{F}\otimes_{\mathcal{O}_{S\times X}}\omega_X)$, such that, for each closed point $s$ of the parameter scheme $S$, the pull-back $(\mathcal{F}_s,\Phi_s)$ of $(\mathcal{F},\Phi)$ to the fiber of the projection $S\times X\to S$ over $s$, is an $H$-semistable Higgs sheaf of class $c$ on $X$. By Proposition \ref{ss implies s}, $\mathcal{F}$ is then a flat family of $H$-semistable sheaves of class $c$ on $X$ parametrized by $S$. Moreover, an isomorphism $(\mathcal{F},\Phi)\xrightarrow{\simeq}(\mathcal{G},\Psi)$ of families of Higgs sheaves (parametrized by the same scheme) gives, in particular, an isomorphism $\mathcal{F}\xrightarrow{\simeq}\mathcal{G}$ of families of  ordinary sheaves. It follows that the association
$$
[(\mathcal{F},\Phi)]\mapsto[\mathcal{F}]
$$
(here $[\ast]$ denotes the isomorphism class of the object $\ast$ in the appropriate category) gives a well-defined set map
$$
\varphi_{\mathrm{Higgs}}(S)\to\varphi(S),
$$
which is clearly natural in $S$. Thus we have a natural transformation of functors $(\mathrm{Sch}/k)^{\mathrm{op}}\to(\mathrm{Set})$
$$
\alpha:\varphi_\mathrm{Higgs}\to\varphi.
$$
Composing $\alpha$ with the natural transformation $\varphi\to\mathrm{h}_{M(c)}$, we obtain a natural transformation
$$
\varphi_{\mathrm{Higgs}}\xrightarrow{\alpha}\varphi\to\mathrm{h}_{M(c)},
$$
which then factors uniquely through $\varphi_{\mathrm{Higgs}}\to\mathrm{h}_{M_{\mathrm{Higgs}}(c)}$, producing a natural transformation
$$
\overline{\alpha}:\mathrm{h}_{M_{\mathrm{Higgs}}(c)}\to\mathrm{h}_{M(c)}.
$$
By Yoneda's Lemma we then have $\overline{\alpha}=\mathrm{h}_p$ for a unique morphism of $k$-schemes
$$
p:M_\mathrm{Hig}(c)\to M(c).
$$

In the same way, sending (the isomorphism class of) a family $\mathcal{F}$ of $H$-semistable sheaves of class $c$ to the (isomorphism class of the) family of ($H$-semistable, of class $c$) Higgs sheaves $(\mathcal{F},0)$, produces a morphism of schemes $M(c)\to M_{\mathrm{Higgs}}(c)$, which is clearly a section of $p:M_\mathrm{Higgs}(c)\to M(c)$, thus proving the surjectivity of $p$.
\end{enumerate}
\end{proof}

Thus, by $(2)$ of  Proposition \ref{res}, the moduli scheme $M_{\mathrm{Higgs}}(c)$ of $H$-semistable Higgs sheaves of class $c$ on $X$ appears as the total space of a fibration with base the moduli scheme $M(c)$ of $H$-semistable ordinary sheaves, also of class $c$. In the case $\mathrm{char}(k)\neq 2$, and under the additional assumption that $\pi$ is elliptic and non-isotrivial, we will now show that this fibration is trivial with fiber an affine space depending only on the surface $X$ (and not on the particular choice of genus one fibration $\pi:X\to B$).

First of all, we recall that an elliptic surface $p:S\to C$ is said to be isotrivial if it is \emph{\'etale-locally trivial}, that is, if there exists a surjective \'etale map $D\to C$ from a curve $D$ to $C$, having the property that the base-changed fibration $D\times_CS\to D$ is trivial, namely, isomorphic to a trivial fibration $D\times E\to D$, for some elliptic curve $E$. The property of non-isotrivial elliptic surfaces we shall use is expressed as follows.

\begin{lem}
\label{lem}
Let $p:S\to C$ be a non-isotrivial elliptic surface. Then the pull-back of the extension
\begin{equation}
\label{estensione}
\mathcal{K}_p:\qquad 0\to p^\ast\omega_C\xrightarrow{p^\ast}\Omega_S\to\Omega_p\to 0
\end{equation}
to the generic fiber of $p$ is a non-split extension of trivial line bundles.

\begin{proof}
{
The claim is equivalent to the assertion that the image 
$$
s\in\mathrm{H}^0(B,\mathit{Ext}^1_p(\Omega_p,p^\ast\omega_C))
$$
of the class of the extension \eqref{estensione} is generically non-zero. Let $U\subseteq C$ be the open (dense) subset of $C$ consisting of regular values of $p$, and let $\hat{p}:p^{-1}(U)\to U$ be the base change of $p$ along the inclusion $U\hookrightarrow C$. Then the restriction $\hat{s}$ of $s$ to $U$ is the same as an element
$$
\hat{s}
\in
\mathrm{H}^0
(U,
\mathit{Ext}^1_{\hat{p}}(\Omega_{\hat{p}},\hat{p}^\ast\omega_U))
\simeq
\mathrm{Hom}(\Theta_U,\mathrm{R}^1\hat{p}_\ast\Theta_{\hat{p}}),
$$
which is nothing but the Kodaira-Spencer section of the smooth morphism $\hat{p}$; this is a morphism of invertible sheaves (the sheaf $\mathrm{R}^1\hat{p}_\ast\Theta_{\hat{p}}$ is invertible by Grauert's theorem), and it is non-zero by virtue of the assumption of non-isotriviality.}
\end{proof}

\end{lem}

The next result is about sheaves of arbitrary rank.

\begin{prop}
\label{la prop}
Let $\pi:X\to B$ be elliptic and non-isotrivial, and let $F$ be a torsion-free sheaf on $X$, semistable on a general closed fiber of $\pi$.
Assume that the rank and fiber-degree of $F$ are relatively prime, and that the rank is non-zero in the ground field $k$. Then the natural map
\begin{equation}
\label{map}
\mathrm{H}^0(\Omega_X)\simeq\mathrm{Hom}(\mathcal{O}_X,\Omega_X)\xrightarrow{-\otimes\mathrm{id}_F}\mathrm{Hom}(F\otimes\mathcal{O}_X,F\otimes\Omega_X)\simeq\mathrm{Hom}(F,F\otimes\Omega_X)
\end{equation}
is an isomorphism of $k$-vector spaces. Thus, each $\Omega_X$-valued field $\phi:F\to F\otimes\Omega_X$ on $F$ is of the form $s\mapsto s\otimes \alpha$ for a unique global $1$-form $\alpha$ on $X$, and, in particular, it satisfies the integrability condition $\phi\wedge\phi=0$.
\end{prop}

\begin{proof}
First of all, there is a well defined $X$-linear \emph{trace morphism} 
$$
\mathrm{tr}_F:\mathit{End}(F)\to\mathcal{O}_X
$$
(let $S$ be the scheme of singularities of $F$, which is a closed subscheme of $X$ of codimension $\geqslant 2$; for a Zariski open $U\subseteq X$, a section $\alpha\in\Gamma(U,\mathit{End}(F))$ restricts to an endomorphism of the locally free sheaf $F|_{U\setminus S}$, whose trace is a well-defined element of $\Gamma(U\setminus S,\mathcal{O}_X)$, and this, in turn, is the restriction of a unique element of $\Gamma(U,\mathcal{O}_X)$). The $X$-linear map 
$$
\mathcal{O}_X\to\mathit{End}(F)
$$
corresponding to the identity endomorphism of $F$ splits $\mathrm{tr}_F$, since $\mathrm{rk}(F)\neq 0$ in $k$. Thus, if we set
$$
\mathit{End}(F)_0:=\ker(\mathrm{tr}_F),
$$
then we have a canonical decomposition
$$
\mathit{End}(F)\simeq\mathcal{O}_X\oplus\mathit{End}(F)_0.
$$
As a consequence, the space
$$
\mathrm{Hom}(F,F\otimes\Omega_X)=\mathrm{H}^0(\mathit{Hom}(F,F\otimes\Omega_X))\simeq\mathrm{H}^0(\mathit{End}(F)\otimes\Omega_X)
$$
splits as
$$
\mathrm{Hom}(F,F\otimes\Omega_X)\simeq\mathrm{H}^0(\Omega_X)\oplus\mathrm{H}^0(\mathit{End}(F)_0\otimes\Omega_X).
$$
We claim that the second summand is zero. In fact, we can write
$$
\mathrm{H}^0(\mathit{End}(F)_0\otimes\Omega_X)=\mathrm{H}^0(\pi_\ast(\mathit{End}(F)_0\otimes\Omega_X)),
$$
where the sheaf $\pi_\ast(\mathit{End}_0(F)\otimes\Omega_X)$ is locally-free (since $\mathit{End}_0(F)\otimes\Omega_X$ is torsion-free). To determine its rank, we observe that, on a general closed fiber $E=X_b$ of $\pi$, $F$ restricts to a stable locally free sheaf $W=F_b$, while, as one easily checks, the sheaf $\mathit{End}(F)_0$ restricts to
$$
\mathit{End}(W)_0:=\ker(\mathrm{tr}_W:\mathit{End}(W)\to\mathcal{O}_E).
$$
Moreover, the bundle $\Omega_X$ restricts on $E$ to the \emph{non-split} self-extension $\mathbb{I}$ of the structure sheaf $\mathcal{O}_E$, since $\pi$ has been assumed to be non-isotrivial.  Thus $\mathit{End}(F)_0\otimes\Omega_X$ restricts to $\mathit{End}(W)_0\otimes\mathbb{I}$, and from the exact sequence
$$
0\to\mathit{End}(W)_0\otimes\mathbb{I}\to\mathit{End}(W)\otimes\mathbb{I}\to\mathbb{I}\to 0
$$
we obtain
$$
\mathrm{H}^0(\mathit{End}(W)_0\otimes\mathbb{I})=\ker\left(\mathrm{tr}:\mathrm{Hom}(W,W\otimes\mathbb{I})\to\mathrm{H}^0(\mathbb{I})\right).
$$
To compute the map $\mathrm{tr}:\mathrm{Hom}(W,W\otimes\mathbb{I})\to\mathrm{H}^0(\mathbb{I})$ it is enough to recall some facts from the theory of sheaves on elliptic curves (defined over an algebraically closed field). First of all,
if
$$
\mathcal{E}:\qquad 0\to\mathcal{O}_E\to\mathbb{I}\to\mathcal{O}_E\to 0
$$
is a defining extension of $\mathbb{I}$, then the map $\mathcal{O}_E\to\mathbb{I}$ induces an isomorphism on global sections. Thus, if $\xi\in\mathrm{H}^0(\mathbb{I})$ is the image of $1\in\mathrm{H}^0(\mathcal{O}_E)$, then $\mathrm{H}^0(\mathbb{I})=k\xi$. Next, the sheaf $W\otimes\mathbb{I}$ is indecomposable (with respect to $\oplus$). In particular, the extension
\begin{equation}
\label{ses5}
W\otimes\mathcal{E}:\qquad 0\to W\xrightarrow{j}  W\otimes \mathbb{I}\xrightarrow{q} W\to 0
\end{equation}
is also non-split. From this it follows easily that the map
$$
j_\ast:\mathrm{Hom}(W,W)\to\mathrm{Hom}(W,W\otimes\mathbb{I}),\qquad\phi\mapsto j\circ \phi
$$
is an isomorphism. In fact, for each map $\phi:W\to W\otimes\mathbb{I}$, the composition $q\circ\phi:W\to W$ is zero (otherwise, it would be a non-zero multiple of $\mathrm{id}_W$, hence an automorphism of $W$, and the sequence \eqref{ses5} would split), showing that $\phi$ factors uniquely through $j$. Thus $\mathrm{Hom}(W,W\otimes\mathbb{I})=kj$. Now,
$$
\mathrm{tr}(j)=\mathrm{rk}(F)\xi\neq 0,
$$
showing that $\mathrm{tr}:\mathrm{Hom}(W,W\otimes\mathbb{I})\to\mathrm{H}^0(\mathbb{I})$ is an isomorphism, whence
$$
\mathrm{H}^0(\mathit{End}(W)_0\otimes\mathbb{I})=0.
$$
So the sheaf $
\pi_\ast(\mathit{End}(F)_0\otimes\Omega_X)
$
is zero, and $\mathrm{H}^0(\mathit{End}(F)_0\otimes\Omega_X)=0$, as claimed.
\end{proof}

In the next result, we use the notation
$$
\mathbb{A}(V):=\spec(\sym V^\vee)
$$
for the affine space associated to a $k$-vector space $V$.

\begin{prop}
\label{res 2}
Let us assume $\mathrm{char}(k)\neq 2$ and $\pi$ elliptic and non-isotrivial.   Let
$$
\mathrm{tr}:M_{\mathrm{Higgs}}(c)\to \mathbb{A}(\mathrm{H}^0(\Omega_X))
$$
be the first component of the \emph{Hitchin fibration} (that is, the morphism sending a Higgs sheaf to the trace of its Higgs field). Then the morphism of $M(c)$-schemes
\begin{equation}
\label{mor2}
(p,\mathrm{tr}):M_{\mathrm{Higgs}}(c)\to M(c)\times\mathbb{A}(\mathrm{H}^0(\Omega_X))
\end{equation}
is an isomorphism. Here $p:M_{\mathrm{Higgs}}(c)\to M(c)$ is the fibration \eqref{forget}.
\end{prop}

\begin{proof}
We will prove this by explicitly constructing an inverse to $(p,\mathrm{tr})$. Set
$$
\mathcal{H}:=\mathrm{H}^0(\Omega_X).
$$
With the notation of the proof of Proposition \ref{res}, it is clear that the scheme $M(c)\times\mathbb{A}(\mathcal{H})$ corepresents the functor
$$
(\mathrm{Sch}/k)^{\mathrm{op}}\to(\mathrm{Set}),\qquad S\mapsto\varphi(S)\times\mathrm{Hom}_{k-\mathrm{lin}}(\mathcal{H}^\vee,\Gamma(S,\mathcal{O}_S)).
$$
Observe that
\begin{eqnarray*}
\mathrm{Hom}_{k-\mathrm{lin}}(\mathcal{H}^\vee,\Gamma(S,\mathcal{O}_S))&\simeq & \mathcal{H}\otimes_k\mathrm{H}^0(S,\mathcal{O}_S)\\
&\simeq &\mathrm{H}^0(S\times X,\Omega_X\otimes\mathcal{O}_{S\times X})\simeq\mathrm{Hom}_{S\times X}(\mathcal{O}_{S\times X},\Omega_X\otimes\mathcal{O}_{S\times X}).
\end{eqnarray*}
Thus we have a map
\begin{eqnarray*}
\varphi(S)\times\mathrm{Hom}_{k-\mathrm{lin}}(\mathcal{H}^\vee,\Gamma(S,\mathcal{O}_S))&\to &\varphi_{\mathrm{Higgs}}(S),\\
([\mathcal{F}],a:\mathcal{O}_{S\times X}\to\Omega_X\otimes\mathcal{O}_{S\times X}) &\mapsto &[(\mathcal{F},\mathrm{id}_{\mathcal{F}}\otimes a)],
\end{eqnarray*}
which is readily seen to be well defined and natural in $S$, thus giving rise to a morphism of schemes $f:M(c)\times\mathbb{A}(\mathcal{H})\to M_{\mathrm{Higgs}}(c)$. Using Proposition \ref{la prop}, one checks that $f$ is the required inverse of $(p,\mathrm{tr})$.

\end{proof}

As a consequence of the isomorphism \eqref{mor2}, one has the following:

\begin{cor}
Let the assumptions be the same as in Proposition \ref{res 2}. Then the moduli scheme $M_{\mathrm{Higgs}}(c)$ is non-singular, and for each closed point $[(F,\phi)]\in M_{\mathrm{Higgs}}(c)$ there is a canonical isomorphism
$$
\mathrm{T}_{M_{\mathrm{Higgs}}(c),[(F,\phi)]}\simeq\mathrm{Ext}^1_X(F,F)\oplus\mathrm{H}^0(\Omega_X).
$$

\end{cor}

\begin{proof}
The scheme $M(c)$, which equals $M^\mathrm{s}(c)$ by Proposition \ref{res}(1), is non-singular, since, for each closed point $[F]\in M(c)$, we have, by Serre duality and the decomposition $\mathit{End}(F)=\mathcal{O}_X\oplus\mathit{End}(F)_0$,
$$
\mathrm{ext}^2(F,F)_0=\mathrm{hom}(F,F\otimes\omega_X)-\mathrm{h}^0(\omega_X)=\mathrm{h}^0(\mathit{End}(F)_0\otimes\omega_X)=0;
$$
in last equality we used the fact that $\mathit{End}(F)_0\otimes\omega_X$ is torsion-free, with
$$
\mathrm{h}^0(\mathit{End}(F)_0\otimes\omega_X\otimes\mathcal{O}_{X_b})=\mathrm{end}(F\otimes\mathcal{O}_{X_b})_0=0
$$
for $b\in B(k)$ general (recall that $\omega_{X_b}\simeq\mathcal{O}_{X_b}$, since $X_b$ is an elliptic curve, and $\mathcal{O}_{X}(-X_b)\simeq\pi^\ast\mathcal{O}_B(-b)$, so that  $\omega_X\otimes\mathcal{O}_{X_b}\simeq\omega_{X_b}\otimes\mathcal{O}_{X}(-X_b)\simeq\omega_{X_b}\otimes(\mathcal{O}_B(-b)\otimes k(b))$ is trivial). Moreover, there is a canonical isomorphism 
$$
\mathrm{T}_{[F]}M(c)\simeq\mathrm{Ext}^1_X(F,F).
$$
Combining this with the isomorphism \eqref{mor2} proves the claims.
\end{proof}

\bigskip
\section{Arbitrary fiber degree}\label{sec:arb}
In this section we denote by $k$ an algebraically closed field of characteristic zero. We will prove two results on Higgs sheaves on elliptic surfaces with section, which suggest a close relationship between Higgs sheaves on $X$ and pairs on $X$ with values in a suitable line bundle. The study of the moduli space of such pairs might be simpler then the study of the Higgs moduli space, and it might shed some light on the properties of the latter.
\subsection{Semistable bundles on elliptic curves}
Let us start by recalling some properties of semistable sheaves on genus one curves. Let $C$ be a non-singular, connected, complete curve of genus $1$. Then, for each positive integer $r$ there exists a unique (up to isomorphism) indecomposable locally-free sheaf $\mathbb{I}_r$ of rank $r$ and degree $0$ on $C$ such that $\mathrm{h}^0(C,\mathbb{I}_r)\neq 0$. Moreover $\mathbb{I}_r$ is (strictly) semistable and $\mathrm{S}$-equivalent to $\mathcal{O}_C^{\oplus r}$ (in particular $\det(\mathbb{I}_r)\simeq\mathcal{O}_C$), and $\mathrm{h}^0(C,\mathbb{I}_r)=1$.

Now let $(r,d)$ be a pair of integers with $r>0$, and set $h:=\textrm{gcd}(r,d)$, and
$$
r':=\frac{r}{h},\qquad d':=\frac{d}{h},
$$
so that $\textrm{gcd}(r',d')=1$. Then any semistable locally-free sheaf of rank $r$ and degree $d$ is $\mathrm{S}$-equivalent to a sheaf of the form
\begin{equation}
\label{she}
\bigoplus_i\mathbb{I}_{r_i}\otimes V_i,
\end{equation}
where each $V_i$ is a stable locally-free sheaf of rank $r'$ and degree $d'$, and the $r_i$'s are positive integers satisfying $\sum r_i=h$; assigning to the sheaf \eqref{she} the effective divisor of degree $h$
$$
\sum_i r_i[V_i]
$$
on the moduli space $\mathrm{M}(r',d')$ gives an isomorphism
$$
\mathrm{M}_C(r,d)\simeq\sym^h\mathrm{M}_C(r',d').
$$
Here we used the notation $\sym^hS:=S^h/\sym_h$ for a scheme $S$, where $\sym_h$ is the symmetric group on $h$ letters, acting on $S^h=S\times\cdots\times S$ by permuting the factors; this is the scheme parametrizing effective divisors of degree $h$ on $S$.

 In addition, each $V_i$ is uniquely determined by its determinant, which is an element of $\mathrm{Pic}^{d'}(C)$. The choice of an origin $\mathbf{o}\in C(k)$ determines an isomorphism
$$
C\xrightarrow{\simeq}\mathrm{Pic}^{d'}_{C/k},
$$
which on closed points reads as
$$
C(k)\to\mathrm{Pic}^{d'}(C),\qquad x\mapsto\mathcal{O}_C(x+(d'-1)\mathbf{o}),
$$
and thus, finally, an isomorphism
\begin{equation}
\label{mod}
\mathrm{M}_C(r,d)\simeq\sym^hC.
\end{equation}
 Sheaves corresponding via \eqref{mod} to reduced divisors on $C$ are actually isomorphic (not just $\mathrm{S}$-equivalent) to sheaves of the form \eqref{she}, with $r_i=1$ for each $i$, and the $V_i$'s pairwise non-isomorphic.

\subsection{Friedman's spectral covers}  Let $\pi:X\to B$ be a genus one fibration with a section $\sigma:B\to X$. Let $F$ be a torsion-free sheaf on $X$ such that its pull-back $F_\eta$ to the generic fiber of $\pi$ is semistable (i.e., such that its restriction $F_b$ to a closed fiber $X_b\hookrightarrow X$ of $\pi$ is semistable for $b\in B(k)$ general); e.g., $F$ might have rank $2$ and admit a Higgs field $\phi$ such that the Higgs sheaf $(F,\phi)$ is slope-semistable with respect to a $\mathrm{c}(F)$-suitable polarization on $X$ (cf. Proposition \ref{gen}). Thus, if we set
$$
r:=\mathrm{rk}(F),\qquad d:=\deg_\pi(F),\qquad h:=\textrm{gcd}(r,d),
$$
we obtain, for $b\in B(k)$ general, a closed point $[F_b]$ of the moduli space $\mathrm{M}_{X_b}(r,d)$, or, using the isomorphism
$$
\mathrm{M}_{X_b}(r,d)\simeq\sym^hX_b
$$
determined by the choice of origin $\sigma(b)\in X_b(k)$ in accordance with \eqref{mod}, an effective divisor of degree $h$ on $X_b$. In other words, $F$ determines a rational section $B\dashrightarrow \sym^h\pi$ of the structure morphism
$$
\sym^h\pi:=X\times_{\pi,B,\pi}\cdots\times_{\pi,B,\pi} X/\sym_h\to B
$$
of the relative $h$-th symmetric power of $\pi$. Since the base $B$ is regular and one-dimensional, and the scheme $\sym^h\pi$ complete, this extends to a global section
\begin{equation}
\label{sect}
\mathrm{s}_F:B\to\sym^h\pi.
\end{equation}
The $B$-scheme $\sym^h\pi$ is a fine moduli scheme parametrizing effective degree $h$ divisors on the fibers of $\pi$; so there is a relative universal effective divisor of degree $h$
$$
\mathcal{D}_h\hookrightarrow X\times_{B}\sym^h\pi,
$$
flat over $\sym^h\pi$. Using the section \eqref{sect} to base-change the finite, degree $h$ morphism
$$
\mathcal{D}_h\hookrightarrow X\times_{B}\sym^h\pi\to\sym^h\pi
$$
we obtain a cartesian square
\begin{center}
\begin{tikzcd}
\mathrm{D}_F\arrow{r}\arrow{d}& \mathcal{D}_h\arrow{d}\\
B\arrow{r}{\mathrm{s}_F}& \sym^h\pi
\end{tikzcd};
\end{center}
we call the morphism
$$
\mathrm{D}_F\to B
$$
the \emph{Friedman spectral cover} associated to $F$. The composite
$$
\mathrm{D}_F\to\mathcal{D}_h\hookrightarrow X\times_{\pi}\sym^h\pi\to X
$$
embeds $\mathrm{D}_F$ in $X$ as an effective divisor, which we call the \emph{Friedman spectral curve of $F$}.

\subsection{The results}
Let $\pi:X\to B$ be an elliptic surface. We will denote by $\mathcal{S}_\pi$ the saturation of $\pi^\ast\omega_B$ in $\Omega_X$, and by $(\Omega_\pi)_{\textrm{t.f.}}$ the quotient of $\Omega_\pi$ by its torsion subsheaf, so that there is an exact sequence
\begin{equation}
\label{satu}
0\to\mathcal{S}_\pi\to\Omega_X\to(\Omega_\pi)_{\textrm{t.f.}}\to 0.
\end{equation}
One sees easily that the sheaf $\mathcal{S}_\pi$ is invertible, and in fact there is an isomorphism
$$
\mathcal{S}_\pi\simeq\mathcal{O}_X(\mathrm{V}_\pi)\otimes\pi^\ast\omega_B,
$$
where $\mathrm{V}_\pi$ is an effective \emph{vertical}\footnote{A divisor on $X$ is vertical (with respect to $\pi$) if its support is a union of irreducible components of fibers of $\pi$.} divisor on $X$; in particular, the restriction of $\mathcal{S}_\pi$ to almost all closed fibers of $\pi$ is trivial. The next result generalizes Proposition 4.2 and Corollary 4.3 of \cite{primo} to the case of sheaves with arbitrary determinant.

\begin{prop}
\label{reduce}
Let $\pi:X\to B$ be a non-isotrivial elliptic surface with a section $\sigma:B\to X$, and let $F$ be a torsion-free sheaf of on $X$. Let us assume that:
\begin{enumerate}
\item the restriction of $F$ to a general closed fiber of $\pi$ is semistable;
\item the Friedman spectral curve $\mathrm{D}_F\hookrightarrow X$ of $F$ is reduced;
\item the schemes of singularities of $F$ and $(\Omega_\pi)_{\mathrm{t.f.}}$\footnote{The scheme of singularities of a torsion-free sheaf $G$ on $X$ is the support of the cokernel of the canonical monomorphism $G\to G^{\vee\vee}$.} are disjoint.
\end{enumerate}
Then the natural map
$$
\mathrm{Hom}(F,F\otimes\mathcal{S}_\pi)\to\mathrm{Hom}(F,F\otimes\Omega_X)
$$
is an isomorphism of $k$-vector spaces. Thus each $\Omega_X$-valued field $\phi:F\to F\otimes\Omega_X$ factors uniquely through $F\otimes\mathcal{S}_\pi\hookrightarrow F\otimes\Omega_X$, and it satisfies the integrability condition $\phi\wedge\phi=0$.

\end{prop}

\begin{proof}
First of all, one sees easily that the sequence
\begin{equation}
\label{tens}
0\to \mathit{End}(F)\otimes \mathcal{S}_\pi\to \mathit{End}(F)\otimes \Omega_X\to \mathit{End}(F)\otimes (\Omega_\pi)_{\textrm{t.f.}}\to 0,
\end{equation}
obtained by tensoring \eqref{satu} with the sheaf $\mathit{End}(F)$, is exact. The sheaves $\mathit{End}(F)\otimes \mathcal{S}_\pi$ and $\mathit{End}(F)\otimes \Omega_X$ are torsion-free, since they are tensor products of the torsion-free sheaf $\mathit{End}(F)$ with locally free ones (recall that $\mathcal{S}_\pi$ is invertible), and assumption $(3)$ implies that $\mathit{End}(F)\otimes (\Omega_\pi)_{\textrm{t.f.}}$ is torsion-free too. Thus, applying the functor $\pi_\ast$ to \eqref{tens}, we obtain an exact sequence of locally free sheaves on $B$
\begin{equation}
\label{on b}
0\to \pi_\ast(\mathit{End}(F)\otimes \mathcal{S}_\pi)\to \pi_\ast(\mathit{End}(F)\otimes \Omega_X)\to \pi_\ast(\mathit{End}(F)\otimes (\Omega_\pi)_{\textrm{t.f.}}).
\end{equation}
The restrictions of the sheaves in \eqref{tens} to a general closed fiber $C=X_b$ of $\pi$ are isomorphic either to $\mathit{End}(V)$, or to $\mathit{End}(V)\otimes\mathbb{I}_2$, for some semistable locally free sheaf $V$ on $C$ of rank $r=\mathrm{rk}(F)$ and degree $d:=\deg_\pi(F)$ (recall that $\mathbb{I}_2$ is the unique non-split self-extension of the structure sheaf of the elliptic curve $C$). Moreover, by  assumption $(2)$, there is an isomorphism
$$
V\simeq\bigoplus_{i=1}^hV_i
$$
where $h:=\textrm{gcd}(r,d)$, and the $V_i$'s are pairwise non-isomorphic stable locally free sheaves of rank $r/h$ and degree $d/h$. Thus
\begin{eqnarray*}
\mathrm{h}^0(C,\mathit{End}(V))=\sum_{i,j=1}^h\mathrm{hom}(V_i,V_j)=\sum_{i,j=1}^h\delta_{ij}=h,\\
\mathrm{h}^0(C,\mathit{End}(V)\otimes\mathbb{I}_2)=\sum_{i,j=1}^h\mathrm{hom}(V_i,V_j\otimes\mathbb{I}_2)=\sum_{i,j=1}^h\delta_{ij}=h,
\end{eqnarray*}
from which it follows that the non-zero sheaves in \eqref{on b} all have the same rank $h$, and so the map
$$
\pi_\ast(\mathit{End}(F)\otimes \mathcal{S}_\pi)\to\pi_\ast(\mathit{End}(F)\otimes \Omega_X)
$$
is an isomorphism. The induced induced map on global sections is thus an isomorphism
$$
\mathrm{Hom}(F,F\otimes\mathcal{S}_\pi)\xrightarrow{\simeq}\mathrm{Hom}(F,F\otimes\Omega_X),
$$
as claimed.
\end{proof}

	{
	
Next, we will prove that a result similar to Proposition \ref{reduce} holds for rank two sheaves also in the case of non-reduced spectral covers, if one assumes in addition the restriction of the sheaf to a general closed   fiber to be \emph{regular}. }
	\begin{prop}
\label{reduced}
Let $\pi:X\to B$ be a non-isotrivial elliptic surface with a section $\sigma:B\to X$, and let $F$ be a torsion-free Higgs sheaf of rank $2$ on $X$. Let us assume that:
\begin{enumerate}
\item the restriction of $F$ to a general closed fiber of $\pi$ is semistable and regular;
\item the Friedman spectral curve $\mathrm{D}_F\hookrightarrow X$ of $F$ is non-reduced;
\item the schemes of singularities of $F$ and $(\Omega_\pi)_{\mathrm{t.f.}}$ are disjoint.
\end{enumerate}
Then the natural map
$$
\mathrm{Hom}(F,F\otimes\mathcal{S}_\pi)\to\mathrm{Hom}(F,F\otimes\Omega_X)
$$
maps $\mathrm{Hom}(F,F\otimes\mathcal{S}_\pi)$ bijectively onto the set of $\Omega_X$-valued fields $\phi:F\to F\otimes\Omega_X$ on $F$ satisfying the integrability condition $\phi\wedge\phi=0$.
\end{prop}

\begin{proof}
Tensoring \eqref{satu} by $F$ and then applying $\mathrm{Hom}(F,-)$ we get an exact sequence
$$
0\to \mathrm{Hom}(F,F\otimes\mathcal{S}_\pi)\to\mathrm{Hom}(F,F\otimes\Omega_X)\to \mathrm{Hom}(F,F\otimes(\Omega_\pi)_{\mathrm{t.f.}}).
$$
Thus it will be enough to show that for each Higgs field $\phi:F\to F\otimes\Omega_X$, the composition
\begin{equation}
\label{com}
F\xrightarrow{\phi} F\otimes\Omega_X\to F\otimes(\Omega_\pi)_{\mathrm{t.f.}}
\end{equation}
vanishes. To this end, it is sufficient to prove that the restriction of \eqref{com}  to a general closed fiber of $\pi$ is zero. So let $E=X_b$ be such a fiber, and let
\begin{equation}
\label{exte}
\mathcal{E}:\qquad 0\to\mathcal{O}_E\xrightarrow{j}\mathbb{I}\xrightarrow{q}\mathcal{O}_E\to 0
\end{equation}
be the unique non-split self-extension of $\mathcal{O}_E$. Then there are isomorphisms
$$
F_b\simeq \mathbb{I}\otimes \lambda,\qquad{(\Omega_X)}_b\simeq\mathbb{I},\qquad {((\Omega_\pi)_{\mathrm{t.f.}})}_b\simeq\mathcal{O}_E
$$
for some $\lambda\in\pic(E)$; so the Higgs field $\phi$ restricts on $E$ to an element
$$
\psi\in\mathrm{Hom}(\mathbb{I}\otimes\lambda,\mathbb{I}\otimes\lambda\otimes\mathbb{I})\simeq\mathrm{Hom}(\mathbb{I},\mathbb{I}\otimes\mathbb{I})
$$
satisfying
$$
\psi\wedge\psi=0\in\mathrm{Hom}(\mathbb{I},\mathbb{I}\otimes\wedge^2\mathbb{I}).
$$
Let
$$
x:=j(1)\in\Gamma(E,\mathbb{I}).
$$
We will show that $\psi$ can be written as $\psi=\alpha\otimes x$ for some $\alpha\in\mathrm{End}(\mathbb{I})$, from which it will follow that
$$
q_\ast\psi=\alpha\otimes q(x)=\alpha\otimes q(j(1))=0.
$$
To this end, let $\mathcal{U}=\{U,V,\dots\}$ be an affine open cover of $E$. Then for each $U\in\mathcal{U}$ the restriction of the extension \eqref{exte} to $U$ splits, so there exists a section $y_U\in\Gamma(U,\mathbb{I})$ such that 
$$
q(y_U)=1\in\Gamma(U,\mathcal{O}_E).
$$
The pair $(x,y_U)$ is a local frame for $\mathbb{I}$ on $U$, and on an overlap $U\cap V$ one has the equality
\begin{equation}
\label{transit}
y_V=y_U+f_{UV}x
\end{equation}
for some $f_{UV}\in\Gamma(U\cap V,\mathcal{O}_E)$; the family $(f_{UV})_{(U,V)\in\mathcal{U}^2}$ is an $\mathcal{O}_E$-valued Cech $1$-cocycle on $E$, whose image in $\mathrm{H}^1(\mathcal{O}_E)\simeq\mathrm{Ext}^1(\mathcal{O}_E,\mathcal{O}_E)$ is the class of the extension \eqref{exte}, hence is non-zero.

The restriction $\psi_U$ of $\psi$ to $U\in\mathcal{U}$ can be written as
$$
\psi_U=\alpha_U\otimes x+\beta_U\otimes y_U
$$
for suitable $\alpha_U,\beta_U\in\Gamma(U,\mathit{End}(\mathbb{I}))$; on an overlap $U\cap V$ one has, using \eqref{transit},
$$
\psi_V=\alpha_V\otimes x+\beta_V\otimes(y_U+f_{UV}x)=(\alpha_V+f_{UV}\beta_V)\otimes x+\beta_V\otimes y_U,
$$
and this equals $\psi_U$ if and only if
$$
\beta_V=\beta_U,\qquad \alpha_V-\alpha_U=-f_{UV}\beta_V.
$$
Thus the $\beta$'s patch, say $\beta_U=\beta|_U$ for some $\beta\in\mathrm{End}(\mathbb{I})$, and the $\alpha$'s satisfy
\begin{equation}
\label{overl}
\alpha_V-\alpha_U=-f_{UV}\beta.
\end{equation}
Taking the trace in \eqref{overl} one gets
$$
\mathrm{tr}(\alpha_V)-\mathrm{tr}(\alpha_U)=-\mathrm{tr}(\beta)f_{UV},
$$
showing that $\mathrm{tr}(\beta)=0\in\mathrm{H}^0(\mathcal{O}_E)\simeq k$ (otherwise $(f_{UV})$ would be the boundary of the $0$-cocycle $(-\mathrm{tr}(\beta)^{-1}\mathrm{tr}(\alpha_U))_{U\in\mathcal{U}}$). Now, it is well known that the $k$-algebra map
$$
k[t]\to\mathrm{End}(\mathbb{I}),\qquad t\mapsto \theta:=j\circ q
$$
induces an isomorphism of $k$-algebras
$$
k[t]/(t^2)\xrightarrow{\simeq}\mathrm{End}(\mathbb{I});
$$
in particular, $\mathrm{End}(\mathbb{I})$ is $2$-dimensional over $k$, with $\mathrm{End}(\mathbb{I})_0:=\ker(\mathrm{tr}:\mathrm{End}(\mathbb{I})\to k)=k\theta$. It follows that $\beta$ is a multiple of $\theta$, say
$$
\beta=\tau\theta
$$
for some $\tau\in k$. We claim that $\tau=0$; if not, then using the condition $\psi\wedge\psi=0$ one would have, for each $U\in\mathcal{U}$,
$$
0=\psi_U\wedge\psi_U=[\alpha_U,\beta]\otimes \left(x\wedge y_U\right),
$$
(here $[\alpha_U,\beta]=\alpha_U\circ\beta-\beta\circ\alpha_U$ is the commutator of the endomorphisms $\alpha_U$ and $\beta$), that is,
$$
0=[\alpha_U,\beta]=[\alpha_U,\tau\theta]=\tau[\alpha_U,\theta],
$$
and thus
\begin{equation}
\label{commut}
[\alpha_U,\theta]=0.
\end{equation}
Let $T$ be the matrix of $\theta$ with respect to $(x,y_U)$, that is $T=\begin{pmatrix}
0 & 1\\
0 & 0
\end{pmatrix}$,  and write $A_U=\begin{pmatrix}
 a_U & b_U\\
 c_U & d_U
\end{pmatrix}$ ($a_U,\dots,d_U\in\Gamma(U,\mathcal{O}_E)$) for that of $\alpha_U$. Then \eqref{commut} is equivalent to $[A_U,T]=0$ (commutator of matrices), which gives
$$
c_U=0,\qquad a_U=d_U,
$$
or, equivalently,
$$
\alpha_U=a_U+b_U\theta.
$$
Substituting this in \eqref{overl} one gets
$$
(a_V-a_U)+(b_V-b_U)\theta=-f_{UV}\tau\theta;
$$
in particular, one has $b_V-b_U=-f_{UV}\tau$, from which it follows that $(f_{UV})$ is the boundary of the $0$-cocycle $(-\tau^{-1}b_U)_{U\in\mathcal{U}}$, a contradiction. Thus $\beta=0$, which, substituted in \eqref{overl}, gives $\alpha_V=\alpha_U$ for each pair $(U,V)\in\mathcal{U}^2$, showing that the $\alpha$'s patch, say $\alpha_U=\alpha|_U$ for some $\alpha\in\mathrm{End}(\mathbb{I})$. In conclusion, one has $\psi=\alpha\otimes x$, as claimed.
\end{proof}

 \end{document}